\documentclass[12pt]{article}
\usepackage{palatino}
\usepackage{amsmath}
\usepackage{amssymb}
\usepackage{amsthm}
\usepackage{url}
\usepackage{latexsym}
\usepackage{amsmath}
\usepackage{euscript}
\usepackage{graphics}

\newcommand\al\alpha
\newcommand\be\beta
\newcommand\de\delta
\newcommand\la\lambda
\newcommand\tha\theta

\newcommand{\Mod}[1]{\ (\mathrm{mod}\ #1)}

\newcommand\iy\infty

\newcommand\bma{\begin{pmatrix}}
\newcommand\ema{\end{pmatrix}}

\def\lcm{\operatorname{lcm}}
\def \<#1,#2>{{\left<#1\mathbin,#2\right>}}

\def\adots{\mathinner{\mkern2mu\raise1pt\hbox{.}
\mkern3mu\raise4pt\hbox{.}\mkern1mu\raise7pt\hbox{.}}}
\def\finpr{\hfill \hbox{

\vrule height 1.453ex  width 0.093ex  depth 0ex \vrule height 1.5ex
width 1.3ex  depth -1.407ex\kern-0.1ex \vrule height 1.453ex  width
0.093ex  depth 0ex\kern-1.35ex \vrule height 0.093ex  width 1.3ex
depth 0ex}}

\newtheorem{theorem}{Theorem}

\newtheorem{lemma}{Lemma}

\newtheorem{definition}{Definition}
\newtheorem{pro}{Proposition}
\everymath{\displaystyle}

\title{\textbf{Some aspects of number theory \\related to phase operators }}

    \author{  F.  Bouzeffour$^{ab}$ and M. Garayev $^{a}$ \\$^{a}$ Department of mathematics, College of Sciences.\\
 King Saud University, P. O Box $2455$ Riyadh $11451$, Saud Arabia. \\$^{b}$ Department of mathematics, College of Sciences.\\  Carthage University,
  Jarzouna 7021, Bizerte, Tunisia.\\
fbouzaffour@ksu.edu.sa, mgarayev@ksu.edu.sa}

\begin{document}
 \maketitle

{\bf Abstract:} We first extend the multiplicativity property of arithmetic functions to the setting of operators on the Fock space. Secondly, we use phase operators to get representation of some extended arithmetic functions by operators on the Hardy space. Finally, we show that radial limits  to the boundary of the unit disc in the Hardy space  is useful in order to go back to the classical arithmetic functions. Our approach can be understudied as a transition from the classical number theory to quantum setting.
\vskip0.2cm

\section{Introduction }
Many aspects of the Hardy space have a physical meaning. In particular, we found in  \cite{Vourdas90} a  characterization of the number-phase statistical of quantum harmonic oscillator by using the inner-outer factorization of the analytic functions in the unit disc. Other important fact  related to the inner-outer factorization is the arithmetic of inner functions which is due to A. Beurling \cite{Beur}.  In the same sprit the interference in phase space between components of quantum superposition states has also some effect on the arithmetic of inner functions.\\An arithmetic function $\alpha$ is said to be multiplicative if for all relatively prime positive integers $n,\,m$,  $$\alpha(nm)=\alpha(n)\alpha(m).$$
The class of arithmetic functions include well-known multiplicative functions such as, the Euler totient function,  denoted by $\varphi$, is defined as the number of positive integers less than and relatively prime to $n$; the M\"{o}bius function $\mu$, which is given by
$$ \mu(n) = \begin{cases}(-1)^{\omega(n)} & \mbox{ if } n \mbox{ is a square-free integer}\\ 0&\mbox{ otherwise, }\end{cases} $$
where $\omega(n)$ is the number of prime factors of $n$. For a more elaborate account on multiplicative functions, we refer the reader to  \cite{NivZucMont1991}.\\In this work we formulate an extension of  arithmetic multiplicative functions to the operators setting. More precisely we consider  functions such that  their domain is the positive integers and their range included in some subset of linear operators on the Fock space of harmonic oscillator algebra \cite{Brian}.\\ Our principal tools  to construct  extended arithmetic functions is the resolution of the identity in the Fock space and the analytic representation in the Hardy space. In particular we show that  arithmetic functions can be obtained by using the Berezin symbol of operators and the radial limit in the unit disc of the extended arithmetic functions. \\ \noindent The Berezin symbol of bounded operator is a covariant symbol of the operator. More precisely, it is an real analytic function on the unit disc that is bounded by the numerical radius of the operator. Often the behavior of the Berezin symbol of an operator provides important information about the operator itself. On the Hardy spaces, the Berezin symbol uniquely determines the operator.\\
The outline of the paper is as follows. In section 2, we explain the connection
between the arithmetic functions and extended arithmetic functions in Fock space. Phase operators, phase states and their analytic representations in the Hardy space and Berezin symbol are explained in section 3. Section 4 studies multiplicativity and asymptotic mulplicativity of arithmetic functions in Hardy space \cite{Rud}. We derive a several identities for arithmetic functions from the radial limit of extended arithmetic functions. We also provide a generalization of the zeta function.
\section{Extended arithmetic functions}
Let us start with the harmonic oscillator algebra $\mathcal{A}$ spanned by three linear operators $a^{-}$, $a^+$ and $N$ satisfying the following commutation relations:
\begin{equation}[a^-, a^+] = 1,\quad [N,a^\pm]=\pm a,\quad (a^-)^\dagger=a,\quad N^\dagger=N. \label{dag}\end{equation}
The Hilbert space $H$ of states is generated by the number
states $|n\rangle$, where $n = 0,\, 1, \,2,\,\dots\,.$  The states are assumed
to be orthonormal, $\langle m|n\rangle=\delta_{m,n}$, and constitute a basis
in $H$. This representation is usually called Fock space.
The operators $a$,  $a^\dagger$ and $N$
satisfying \eqref{dag} may be realized
in Fock space as
\begin{align}
&a |n\rangle= \sqrt{n}|n\rangle\quad a^\dagger|n=\rangle \sqrt{n+1}|n+1\rangle, \quad N|n\rangle= n|n\rangle.
\end{align}
 In this case the states $\{|n\rangle\}_{n=0}^\infty$ are orthonormal and provide a resolution of the identity
\begin{equation}
\sum_{n=0}^\infty |n\rangle\langle n|=1.\label{reso}
\end{equation} Thus, for every integer $n$ we have the following spectral decomposition
\begin{equation}
\sum_{j=0}^{n-1}\Pi_j(n)=1, \quad n=1,\,2\,\,\dots\,.
\end{equation}
where
\begin{equation}
\Pi_j(n)=\sum_{k=0}^\infty|nk+j\rangle\langle nk+j|.\label{porj1}
\end{equation}
We have also
\begin{equation}
\Pi_i(n)\Pi_j(n)=\delta_{ij}\Pi_j(r) .
\end{equation}
Thus, for every fixed integer $n$, $\{\Pi_j(n)\}_{j=0}^{n-1}$ constitutes a complete system of orthogonal projections in $\mathcal{H}$. We also introduce another operator $S_n$, which we call rotated operator and which we define as
\begin{equation}
S_n=\sum_{s=0}^{n-1}\varepsilon_n^{s}\Pi_s(n), \,\,\varepsilon_n=e^{\frac{2i\pi}{n}}.\label{rot1}
\end{equation}
Using the identity
\begin{equation}
\frac{1}{n}\sum_{s=0}^{n-1}\varepsilon_n^{s(k-l)}=\delta_{kl},
\end{equation}
we can invert \eqref{rot1} and get
\begin{equation}
\Pi_j(n)=\frac{1}{n}\sum_{k=0}^{n-1}\varepsilon_n^{-kj}S_n^k.
\end{equation}
It is easy now to show that the rotated operator $S_r$ is an unitary operator on $\mathcal{H}$ and  satisfies the relation
$$S_n^n=1.$$
From the well-know identity
\begin{equation}
\sum_{k=0}^{n-1}\varepsilon^{jk}_n=\left\{
  \begin{array}{l l}
    1\quad\, \text{if}\quad j|n, \\
   0\quad\, \text{otherwise},
  \end{array} \right.  \label{S1}
\end{equation}
we can also show that for every arithmetic progression $j+n,\dots j+rn$, we have
$$\Pi_{j}(n)=\sum_{k=1}^{r}\Pi_{j+kn}(nr).$$

\begin{lemma}[The Chinese Remainder Theorem] The system of congruences
\begin{equation*} \left\{
  \begin{array}{l l}
    x\equiv a \Mod{n} \\
   x\equiv b \Mod{m}.
  \end{array} \right.
  \end{equation*}
is solvable if, and only if, $\,\gcd(n,m)|a-b.\,$ Any two solutions of the system are incongruent mod $( \lcm(n,m))$.
\end{lemma}
\begin{theorem}\label{th1}
For arbitrary positive integers $n$, $m,$ and $k,l=0,\,1,\,\dots$, we have
\begin{equation}
\Pi_{k}(n)\Pi_{l}(m)= \left\{
  \begin{array}{l l}
    \Pi_{j}(\lcm(n,m))\quad \text{if}\quad \gcd(n,m)\mid l- k \\
    0\quad \text{otherwise}.
  \end{array} \right.
\end{equation}
where $j$ is the unique solution $\Mod{\lcm(n,m)}$ of the system of congruences
\begin{equation} \left\{
  \begin{array}{l l}
    j\equiv\, k \Mod{n} \\
   j\equiv l \Mod{m}.
  \end{array} \right.
  \end{equation}
Here $\lcm(n,m)$ is the least common multiple of  integers $n,\,m.$
\end{theorem}
\begin{proof}Observe that
\begin{equation}
\Pi_j(n)=\sum_{k\equiv j\Mod{n}}|k\rangle\langle k|.
\end{equation}
Then \begin{equation}
\Pi_k(n)\Pi_l(m)=\sum_{\substack{r\equiv k\Mod{n}\\r\equiv l\Mod{m}}}|r\rangle\langle r|.
\end{equation}
The result follows from the Lemma 1.
\end{proof}
\begin{definition}
A sequence $\{\Phi(n)\}_{n=1}^\infty$ of linear operators on the Fock space $\mathcal{H}$ is said to be  extended multiplicative function  on  $\mathcal{H}$, if the  map
$\Phi \colon \mathbb{N}\times \mathcal{H}\to \mathcal{H}$
satisfies the following properties: \\
(1) for any $n \in \mathbb{N} $, the map $\Phi (n)\colon \mathcal{H}\to \mathcal{H}$, $v \to \Phi(n,v)$ is linear,\\
(2) $ \Phi (nm)=\Phi (n)\Phi (m), \,\, \text {whenever}\,\quad \gcd(n,m)=1.$
\end{definition}
Note that every multiplicative function $\alpha: \mathbb{N} \rightarrow \mathbb{C}$ can be identified with the extended multiplicative function $\alpha \, 1$ ($1$ is the identity operator). From Theorem 1 we deduce that for every relatively prime  $n$ and $m$ we have
\begin{equation}
\Pi_{j}(n)\Pi_{j}(m)=\Pi_{j}(nm).\label{bb2}
\end{equation}
Thus shows that for every fixed integer $j$ the sequence $\{\Pi_{j}(n)\}_{n=1}^\infty$ is an extended multiplicative function on  $\mathcal{H}$.
Note also that if $n\mid m$, then
\begin{equation*}
\Pi_{j}(n)\Pi_k(m)=\left\{
  \begin{array}{l l}
    \Pi_k(m),\,\,\, \text{if} \,\,\,k\equiv j \Mod{n} \\
   0,\,\,\,\text{otherwise} .
  \end{array} \right.
  \end{equation*}
By unique factorization, there is a unique way of writing $n=p_1^{\alpha_1}\dots p_k^{\alpha_k}$
and \eqref{bb2} gives that
\begin{equation}
\Pi_{j}(n)=\Pi_{j_1}(p^{\alpha_1})\dots \Pi_{j_k}(p^{\alpha_k}),
\end{equation}
where
\begin{equation}
1\leq j_l\leq p_l^{\alpha_l}\,\,\text{and}\,\,j_l\equiv j\mod p_l^{\alpha_l},\quad l=1,\dots,k.
\end{equation}

 We naturally  extend the convolution products, such as the Dirichlet product, $\lcm$ product and unitary product (see, \cite{Bus}) as follows:
\begin{align}
&(A\, * \,B)(n)=\sum_{kl=n}A(k)B(l)\quad (\text{Dirichlet product}), \label{conv}\\
&(A\, \Box\, B)(n)=\sum_{\lcm(k,l)=n}A(k)B(l)\quad (\text{lcm-product)},\label{conv1}\\
&  (A\, \sqcup\, B)(n)=\sum_{kl=n\,\,\gcd(k,l)=1}A(k)B(l)\quad (\text{unitary product)},\label{conv11}
\end{align}
where $\{A(n)\}$ and $\{B(n)\}$ are two sequences of linear operators on $\mathcal{H}.$\\
Note that  extended convolution products defined above are associative, but in general are not commutative.\\ We can also show that if
$\alpha,\,\beta: \,\mathbb{N}\rightarrow \mathbb{C}$ are two arithmetic functions and $j\geq 1,$ then
\begin{align}
&\alpha \Pi_j\,\square \,\beta \Pi_j=(\alpha \square \beta) \Pi_j,\label{b1}\\& \alpha P_j\sqcup \beta \Pi_j=(\alpha \sqcup \beta)\, \Pi_j, \label{b2}\\&
\nu_0* (\alpha \square \beta )\Pi_j=(\nu_0*\alpha \Pi_j)\,(\nu_0*\beta \Pi_j)\label{b3}.
\end{align}
In particular,
\begin{align}
(\Pi_j\,\square \,\Pi_{j})(n)=M_2(n)\Pi_{j}(n)\quad \mbox{and}\quad (\Pi_j\sqcup \Pi_{j})(n)=2^{\omega(n)}\Pi_{j}(n),
\end{align}
where
\begin{equation*}
M_s(n)=\left\{
  \begin{array}{l l} 1
    \quad \text{if}\,\,\,n=1, \\
    \prod_{k=1}^r\big((a_s+1)^s-a_k^s\big)\quad\,\text{if}\quad n=\prod_{k=1}^rp_k^{a_k},
  \end{array} \right.  \label{S1}
\end{equation*}
and $\nu_0(n)=1.$
\section{Phase operator}
We denote by  $H^2,$ the Hardy space of all analytic functions on the unit disc $\mathbb{D}=\{z\in \mathbb{C}: |z|<1\}$ for which
\begin{equation}
\|f\|^2_2=\sup_{0<r<1}\frac{1}{2\pi}\int_0^{2\pi}|f(re^{i\theta})|^2d\theta <\infty.
\end{equation}
and $H^\infty$ is the space the bounded analytic functions on $\mathbb{D}$ equipped with the
usual norm $$\|f\|_\infty  := \sup_{z\in \mathbb{D} }|f(z)|.$$ Let us recall some fundamental theorems on the Hardy classes see \cite{Rud}. For $f\in H^2$ the radial limit
\begin{equation}
f(e^{i\theta})=\lim_{r\rightarrow 1}f(re^{i\theta})
\end{equation}
exists almost everywhere (a.e.) on the unit circle $\mathbb{T}=\partial \mathbb{D}$. As a consequence, we identify functions $f$ in
$H^2$ with their non-tangential boundary limits on $\mathbb{T}$ also denoted by $f$.
The family $\{e_n(z)=z^n\}_{n=0}^\infty$ is an orthonormal basis for $H^2$. The  reproducing kernel of  $H^2$ is given by
\begin{align}
&k_\lambda(z)=
\frac{1}{1-\overline{\lambda}z}.
\end{align}
Recall that the Berezin symbol $\widetilde{A}$ of a bounded linear operator $A$ on $H^2$ is given by the formula \cite{Zhu}
\begin{equation}
\widetilde{A}(\lambda):=<A\widehat{k}_{\lambda},\widehat {k}_{%
\lambda}>,\text{ }\lambda\in\mathbb{D}.\label{ber1}
\end{equation}
where \begin{equation}
\widehat{k}_{\lambda}=\frac{k_{\lambda}}{\|k_{\lambda}\|} ,\la\in \mathbb{D}.
\end{equation}
Phase operators were defined in \cite{Vourdas90} through the polar decomposition of the annihilation and creation operators $a^-$ and $a^+$
\begin{equation}
a^-=E_-N^{1/2},\quad a^+=N^{1/2}E_+.
\end{equation}
The phase operators $E_+$, $E_-$ can be written as
\begin{align}
&E_+=\sum_{n=0}^\infty |n+1\rangle \langle n|,\\&
E_-=E_+^\dagger=\sum_{n=0}^\infty |n\rangle \langle n+1|.
\end{align}
$E_+$ is an isometric operator but it is not unitary
\begin{equation*}
E_+E_-=1-|0\rangle\langle0|,\quad E_-E_+=1.
\end{equation*}
Phase states and phase operators has been intensively studied in the context of compact and noncompact groups, see \cite{Perelomov86}. The major difficulty in formulating in a consistent way a unitary phase operator for a quantum oscillator is the infinite character of the spectrum of the number operator.
The non-unitarity of the phase operator for quantum
harmonic oscillator is intimately related to the fact that the number operator
spectrum is lower bounded see \cite{les2Berndt}.\\
The eigenstate  of the operator $E_-$ are known to be \cite{Vourdas90}
\begin{equation}
|z\rangle=(1-|z|^2)^{1/2}\sum_{n=0}^\infty z^n|n\rangle=(1-|z|^2)^{1/2}(1-zE_+)^{-1}|0\rangle,\quad |z|<1,
\end{equation}
where
\begin{equation}
(1-zE_+)^{-1}=\sum_{n=0}^\infty (zE_+)^n,\quad |z|<1.
\end{equation}
The overlap of two phase states is
\begin{equation}
\langle z|\la\rangle=\frac{(1-|z|^2)^{1/2}(1-|\la|^2)^{1/2}}{1-\overline{\la}z}.
\end{equation}
Following \cite{Vourdas90}, the analytic representation is defined by mapping the number state $|n\rangle$ into $z^n$. Consequently,  if $|f\rangle$ is an arbitrary state such that
\begin{equation}
|f\rangle=\sum_{n=0}^\infty f_n\,|n\rangle,\quad \sum_{n=0}^\infty |f_n|^2<\infty,
\end{equation}
then it is represented in the Hardy space by the function
\begin{equation}
f(z)=(1-|z|^2)^{-1/2}\langle z^*|f\rangle =\sum_{n=0}^\infty f_nz^n \in H^2\label{iden1}.
\end{equation}
An operator $A$ on the Fock space is represented by
$$(1-|z|^2)^{-1/2}\langle z^*|A|f\rangle.$$
In particular, the phase operators $E_-$ and $E_+$ are represented  by the backward shift operator and the shift operator:
\begin{align*}
&(1-|z|^2)^{-1/2}\langle z^*|E_-|f\rangle:=\frac{f(z)-f(0)}{z},\\&
(1-|z|^2)^{-1/2}\langle z^*|E_+|f\rangle:=zf(z).
\end{align*}
The scalar product of two states $|f\rangle$ and $|g\rangle$ is given by the boundary functions as
\begin{equation}
\langle f|g\rangle=\frac{1}{2\pi}\int_0^{2\pi}f(e^{i\theta})g(e^{i\theta})\,d\theta.
\end{equation}
The phase representation is based on the phase states
\begin{equation}
|\theta\rangle=\lim_{|z|\rightarrow 1}(1-|z|^2)^{-1/2}|z\rangle =\sum_{n=0}^\infty e^{in\theta}|n\rangle, \quad z=|z|e^{i\theta}.\label{iden2}
\end{equation}
In view of these representations \eqref{iden1} and \eqref{iden2}, we will mostly not distinguish between operator on Fock space and it's representation in the  Hardy space.
The Berezin symbol of a bounded operator $A$ on the Fock space   defined in \eqref{ber1} in terms of the above representation by
\begin{equation}
\widetilde{A}(\lambda)=\langle \la^*|A|\la^*\rangle.
\end{equation}
\section{Radial limit of extended arithmetic functions}
In this section we study the radial limit of the extended arithmetic functions.  We need the following definitions.
\begin{definition} Let $\{\Phi(n)\}_{n=1}^\infty$ be a sequence of bounded linear operators on $H^2$.\\
(1)The sequence  $\{\Phi(n)\}_{n=1}^\infty$ is said to be multiplicative if for  all $\la \in \mathbb{D}$ we have
\begin{equation}
\widetilde{\Phi(nm)}(\la)=\widetilde{\Phi(n)}(\la)\widetilde{\Phi(m)}(\la)\quad\text {whenever}\,\quad \gcd(n,m)=1.
\end{equation}
(2)The sequence  $\{\Phi(n)\}_{n=1}^\infty$  is said to be asymptotically multiplicative if \begin{equation}
\lim_{|\la| \to 1}(\widetilde{\Phi(nm)}(\la)-\widetilde{\Phi(n)}(\la)\widetilde{\Phi(m)}(\la))=0\quad\text {whenever}\,\quad \gcd(n,m)=1.
\end{equation}
\end{definition}
It is easy to see that every multiplicative sequence is asymptotically multiplicative. The converse is not true. For example, the sequence $\{\Pi_j(n)\}_{n=1}^\infty$ is asymptotically multiplicative, but not multiplicative. Indeed, a straightforward computation shows that
\begin{align}
\widetilde{\Pi_{j}(n)}(\lambda)=\frac{1-|\lambda|^2}
{1-|\lambda|^{2n}}|\lambda|^{2j}\quad \text{and}\quad \lim_{|\la| \to 1 }\widetilde{\Pi_{j}(n)}(\lambda)=\frac{1}{n}.\label{T1}
\end{align}
More examples of asymptotically multiplicative functions can be obtained by considering the Dirichlet convolution product $\al*\beta \Pi_j$ for arbitrary multiplicative functions $\al$ and $\beta$. From \eqref{T1}, we obtain
\begin{equation}
\widetilde{\al*\beta \Pi_j(n)}(\la)=\sum_{d|n}\al(n/d)\beta(d)\frac{1-|\lambda|^2}
{1-|\lambda|^{2d}}|\lambda|^{2j}.
\end{equation}
In particular, it was shown in \cite{Bouz3} that \begin{equation*}
C_j=\mu\,*\nu_1\Pi_j \quad \mbox{and} \quad T_j=\nu_0\,*\mu \Pi_j,
\end{equation*}
where
\begin{align}
&C_{j}(n)=\sum_{\substack{\gcd(k,n)=1\\1\leq k\leq n}}\varepsilon_{n}^{-kj}S_n^k\label{reff1}\\&
T_{j}(n)=\sum_{\substack{\gcd(k,n)=1\\1\leq k\leq n}}\varepsilon_{n}^{kj}\Pi_{j+k}(n).\label{reff2}
\end{align}
Hence  \begin{equation*}
\widetilde{C_j(n)}=\sum_{d|n}\mu(d)\frac{n}{d}\frac{1-|\lambda|^2}
{1-|\lambda|^{2n/d}}|\lambda|^{2j} \quad \mbox{and} \quad \widetilde{T_j(n)}(\la)=\sum_{d|n}\mu(d)\frac{1-|\lambda|^2}
{1-|\lambda|^{2d}}|\lambda|^{2j}.
\end{equation*}
and
 \begin{equation*}
\lim_{\la \to \partial \mathbb{D}}\widetilde{C_j(n)}=\sum_{d|n}\mu(d)=\epsilon(n) \quad \mbox{and} \quad \widetilde{T_j(n)}(\la)=\sum_{d|n}\frac{\mu(d)}{d}=\nu_0*\nu_1\mu.
\end{equation*}
Taking the Berezin symbol in \eqref{reff1} and \eqref{reff2} we get the following identities
\begin{align}
&\sum_{d|n}\mu(d)\frac{n}{d}\frac{|\lambda|^{2j}}
{1-|\lambda|^{2n/d}}=\sum_{\substack{\gcd(k,n)=1\\1\leq k\leq n}}\frac{\varepsilon_{n}^{-kj}
}{1-\varepsilon^k_n
|\lambda|^{2}}\label{ref1}\\&
\sum_{d|n}\mu(d)\frac{|\lambda|^{2j}}
{1-|\lambda|^{2d}}=
(1-|\lambda|^{2n})^{-1}\sum_{\substack{\gcd(k,n)=1\\1\leq k\leq n}}\varepsilon_{n}^{kj}|\lambda|^{2k+2j}.\label{ref2}
\end{align}
Let $\al$ be an  arbitrary  arithmetic function. We consider the following truncated operator
\begin{equation*}
N_{\al,j}=\sum_{k=1}^\infty \al(k)\overline{\Pi}_j(k),
\end{equation*}
where \begin{equation}
\overline{\Pi}_j(k)=\Pi_{j}(k)-|j\rangle\langle j|=\sum_{l=1}^\infty |lk+j\rangle\langle lk+j|.\label{ef1}
\end{equation}
This operator  acts on the states $|n\rangle$ as
\begin{align}
&N_{\alpha,j}|n\rangle= 0\quad \text{for}\quad n=0,\dots,\,j,\\&N_{\alpha,j}|n\rangle= (\nu_0*\alpha)(n-j)|n\rangle\quad \text{for}\quad n\geq j+1.\label{num1}
\end{align}
Hence we have
\begin{equation}
N_{\al,j}=\sum_{k=1}^\infty (\nu_0*\al)(k)|k+j\rangle\langle k+j|.
\end{equation}
\begin{pro}For every arithmetic functions $\al$ ,$\beta$ and every integer $j$ we have\\
1. $N_{\alpha,j}N_{\beta,j}=N_{\al\Box \beta,j},$\\
2. $N_{\mu*\al,j}N_{\mu*\be,j}=N_{\mu*\al \be,j}.$
\end{pro}
\begin{proof}From \eqref{ef1} and Theorem 1, we have
\begin{align*}
\overline{\Pi}_{j}(n)\overline{\Pi}_{j}(m)&=\big(\Pi_{j}(n)-|j\rangle\langle j| \big)\big(\Pi_{j}(m)-|j\rangle\langle j|\big)\\&=\Pi_{j}(n)\Pi_{j}(m)-|j\rangle\langle j|
\\&=\Pi_{j}(\lcm(n,m))-|j\rangle\langle j|\\&=
    \overline{\Pi}_{j}(\lcm(n,m)).
\end{align*}
Using formula \eqref{b1}, we get
\begin{align*}
N_{\alpha \Box\beta,j}&=\sum_{n=1}
^\infty (\alpha \Box \beta)(n)\overline{\Pi}_j(n)\\&=
\sum_{n=1}
^\infty (\alpha \overline{\Pi}_j\Box  \beta \overline{\Pi}_j)(n)\\&=
\sum_{n=1}
^\infty \alpha(n) \overline{\Pi}_j(n)\sum_{n=1}^\infty \beta (n)\overline{\Pi}_j(n)
\\&=N_{\alpha,j}N_{\beta,j}.\end{align*}The second identity follows from the fact that
$$(\mu*\al)\Box(\mu*\be)=\mu*(\al\be).$$
\end{proof}
In particular, for $\al=\varphi$ ($\varphi$ is the Euler Totient arithmetic function ) and  from the Euler identity
\begin{equation}
n=\sum_{d|n}\varphi(d),
\end{equation} and formula \eqref{num1}, we see that the operator $N_{\varphi,j}$ coincides with the number operator $N$, i.e
\begin{align}
&N_{\varphi,0}=a^+a^-=\sum_{n=1}^\infty \varphi(n)(\Pi_0(n)-|0\rangle \langle0 |),
\\&N_{\varphi,1}=a^-a^+=\sum_{n=1}^\infty \varphi(n)(\Pi_1(n)-|0\rangle \langle0 |).
\end{align}
More generally, the operator $N_{\al,j}$  can be identified with Hamiltonian associated to some deformed oscillator algebra.
The simple choose $\al(n)=\frac{1}{n^s}$ leads to the  following generalization of Riemann zeta function
\begin{equation*}
N_{\al,0}=\sum_{n=1}^\infty \frac{1}{n^s}\overline{\Pi}_{0}(n).
\end{equation*}
It is easy to see that for $\Re(s)>1$, $N_{\al,0}$ is bounded operator in $H^2$ and
$$\|N_{\al,0}\|\leq \zeta(\Re(s)).$$
Furthermore, from \eqref{num1} we have
\begin{equation}
N_{\al,0}|n\rangle=\frac{\sigma_s(n)}{n^s}\,|n\rangle,\label{dir3}
\end{equation}
where the sum of positive divisors function $\sigma_s(n),$ for a real or complex number $s$, is defined by \begin{equation}\sigma_s(m)=\sum_{n|m} n^s.
\end{equation}
The Berezin symbol of $N_{\al,0}$ takes the form
\begin{equation}
\widetilde{N_{\al,0}}(\la)=\sum_{n=1}^\infty \frac{1}{n^s}\frac{1-|\lambda|^2}
{1-|\lambda|^{2n}}|\lambda|^{2n}.\label{dir4}
\end{equation}
Using the inequality
\begin{equation}
|\frac{1}{n^s}\frac{1-|\lambda|^2}
{1-|\lambda|^{2n}}|\lambda|^{2n}|\leq \frac{1}{n^{\Re(s)+1}},\quad \la\in \mathbb{D},
\end{equation}
we obtain the following Lambert convergence
\begin{equation}
\lim_{|\la|\rightarrow 1}\widetilde{N_{\al,0}}(\la)=\zeta(s+1).\end{equation}
We can also show that  \begin{align}
\zeta(s)N_{\al,0}=\sum_{n=1}^\infty \frac{1}{n^s}T_0(n),\quad
\frac{1}{\zeta(s+1)}N_{\al,0}=\sum_{n=1}^\infty \frac{1}{n^{s+1}}C_0(n).
\end{align}
More generally, the boundary limit of Berezin symbol of the operator $N_{\al,j}$ for any  arithmetic function $\al$ coincides with Abel convergence and Lambert convergence. Moreover, the obtained zero limit is closely related with compactness of operator (see \cite{Kara2} and references therein).
\section{Conclusion}
In this work we provide an extension of the classical arithmetic functions and convolution products to the arithmetic quantum theory.
Most importantly, we define a generalized number operator, which can be interpreted as a bosonic Hamiltonian of some deformed Heisenberg oscillator algebra and also it is related to fermionic and parafermionic thermal partition functions.
On the technical side, the most significant result of this paper is the use of phase operators to get connection between the arithmetic of inner functions and arithmetic quantum theories. We have found that these connections are even more compelling and worthy of study.


\begin{thebibliography}{99}
\bibitem{Beur} A. Beurling, On two problems concerning linear transformations in Hilbert
space, Acta Math. 81 (1948), 17 pp. MR0027954 (10:381e)

\bibitem{Bus}
R. G. Buschman, lcm-products of number-theoretic functions
revisited. Kyungpook mathematical journal(1999), 39(1), 159--159.
\bibitem{McC1986} P.~J.~McCarthy, {\it Introduction to Arithmetical Functions}, Uni\-ver\-si\-text, Springer, 1986.
\bibitem{Brian}Hall, Brian C. (2013), Quantum Theory for Mathematicians, Graduate Texts in Mathematics, 267, Springer, ISBN 978--1461471158
\bibitem{NivZucMont1991} I.~Niven, H.~S.~Zuckerman, H.~L.~Montgomery {\it An Introduction to the Theory of Numbers}, 5th edition, John Wiley \& Sons, 1991.
\bibitem{Ram1918} S.~Ramanujan, On certain trigonometrical sums and their applications in the theory of numbers, {\it Trans. Cambridge Philos. Soc.} {\bf 22} (1918), 259--276 (Collected Papers, Cambridge 1927, No. 21).
\bibitem{leh}D. H. Lehmer, On a theorem of von Sterneck, Bull. Amer. Math. Soc., 37 (1931), no. 10, 723--726.

\bibitem{Bouz1}F. Bouzeffour, M. Garayev,
Fractional supersymmetry algebra and lacunary Hermite polynomials. arXiv:1810.08275
\bibitem{Bouz2}F. Bouzeffour, M. Garayev, Supersymmetric Quantum mechanics on the radial
lines, arXiv:1811.02151
\bibitem{Bouz3}F. Bouzeffour, W.Jedidi, M. Garayev, Extended Arithmetic functions. Accepted for publication in The Ramanujan
Journal DOI: 10.1007/s11139-018-0122-8

\bibitem{Rud} W. Rudin, Functional Analysis, McGraw--Hill, 1991.
\bibitem{Pegg}
Pegg D T and Barnett S M 1989 {\it Phys. Rev.} A {\bf 39} 1665
\bibitem{Kara1}M. T. Karaev: Berezin symbol and invertibility of operators on the functional Hilbert spaces. J. Funct. Anal. 238 (2006), 181--192. DOI 10.1016/j.jfa.2006.04.030, MR 2253012, Zbl 1102.47018
    \bibitem{Kara2} M. T. Karaev, S. Saltan: Some results on Berezin symbols. Complex Variables, Theory Appl. 50 (2005), 185-193. DOI 10.1080/02781070500032861 | MR 2123954 | Zbl 1202.47031
\bibitem{Vourdas90}

Vourdas A 1990 {\it Phys. Rev.} A {\bf 41} 1653

Vourdas A, Brif C and Mann A 1996
{\it J. Phys. A: Math. Gen.} {\bf 29} 5887


Vourdas A 1993 {\it Phys. Scr.} {\bf 48} 84

\bibitem{BarutGi}
Barut A. O. and Girardello L., 1971 {\it Commun. Math. Phys.} {\bf 21} 41

\bibitem{Perelomov86}
Perelomov A M 1986 {\it Generalized Coherent States and Their Applications} (Berlin: Springer)

\bibitem{GazeauK}
Gazeau J-P and Klauder J R 1999 {\it J. Phys. A: Math. Gen.} {\bf 32} 123

\bibitem{Antoine}
Antoine J-P, Gazeau J-P, Monceau P, Klauder J R and Penson K A 2001 {\it J. Math. Phys.} {\bf  42} 2349

\bibitem{Reed}
Reed M and Simon B 1978 {\it Methods of modern mathematical physics,
Analysis of operators} vol 4 (New York: Academic Press).
\bibitem{Ivanovic}
Ivanovi\'c I D 1981 {\it J. Phys. A: Math. Gen.} {\bf 14} 3241

\bibitem{WoottersFields}
Wootters W K and Fields B D 1989 {\it Ann. Phys. (N Y)} {\bf 191} 363

\bibitem{les2Berndt}
Berndt B C and Evans R J 1981
{\it Bull. Am. Math. Soc.} {\bf 5} 107

Berndt B C, Evans R J and Williams K S 1998
{\it Gauss and Jacobi Sums} (New York: Wiley)

\bibitem{lesquatre}
Hannay J H and Berry M V 1980 {\it Physica D} {\bf 1} 267

Matsutani S and \^Onishi Y 2003 {\it Found. Phys. Lett.} {\bf 16} 325

Rosu H C, Trevi\~no J P, Cabrera H and Murgu\'ia J S 2006
{\it Int. J. Mod. Phys.} B {\bf 20} 1860

Merkel W, Crasser O, Haug F, Lutz E, Mack H, Freyberger M, Schleich W P,
Averbukh I, Bienert M, Girard B, Maier H and Paulus G G 2006
{\it Int. J. Mod. Phys.} B {\bf 20} 1893

\bibitem{29W81}
Witten E 1981 {\it Nucl. Phys. B} {\bf 188} 513



Cooper F, Khare A and Sukhatme U 2001 {\it Supersymmetry in Quantum Mechanics}
(Singapore: World Scientific)



\bibitem{27Infeld}
Infeld L and Hull T E 1951 {\it Rev. Mod. Phys.} {\bf 23} 21


\bibitem{MD28Faddev}
Faddeev L D 1963 {\it J. Math. Phys.} {\bf 4} 72

\bibitem{MD30Pursey}
Pursey D L 1986 {\it Phys. Rev.} D {\bf 33} 2267

\bibitem{MD29Samsonov}
Samsonov B F 2000 {\it J. Phys. A: Math. Gen.} {\bf 33} 591

\bibitem{MD31Kumar}
Kumar M S and Khare A 1996 {\it Phys. Lett.} A {\bf 217} 73

\bibitem{russes}
Bagrov V G and Samsonov B F 1997 {\it Phys. Part. Nucl.} {\bf 28} 374

\bibitem{MD32Fernandez}
Fernandez C D J, Nieto L M and Rosas-Ortiz O 1995 {\it J. Phys. A: Math. Gen.} {\bf 28} 2693

Fernandez C D J and Hussin V 1999 {\it J. Phys. A: Math. Gen.} {\bf 32} 3603

Carballo J M, Fernandez C D J, Negro J and Nieto L M 2004 {\it J. Phys. A: Math. Gen.} {\bf 37} 10349

\bibitem{MD35Angelova}
Angelova M and  Hussin V 2008 {\it J. Phys. A: Math. Gen.} {\bf 41} 304016

\bibitem{MD36Bagrov}
Bagrov V G and Samsonov B F 1996 {\it J. Phys. A: Math. Gen.} {\bf 29} 1011

\bibitem{Weigert}
Weigert S and Wilkinson M 2008 {\it Phys. Rev.} A {\bf 78} 020303
\bibitem{Zhu}1. K. Zhu (2007). Operator Theory in Function Spaces, Second Edition,
Mathematical Surreys and Monographs, Vol. 138, American Mathematical Society, Providence, R.I

\end{thebibliography}
\end{document}